\documentclass[a4paper]{article}

\usepackage{amssymb,amsthm,amsmath}
\usepackage{bbm}
\usepackage{url}
\usepackage{hyperref}

\usepackage[round]{natbib}
\usepackage[width=13cm]{geometry}

 \newcommand{\N}{\mathbb{N}}      
 \newcommand{\R}{\mathbb{R}}
 
 \newcommand{\Ss}{\mathbb{S}}
 \renewcommand{\d}{\mathrm d}

  \theoremstyle{plain}
  \newtheorem{thm}{Theorem}[section]
  
  \newtheorem{lem}[thm]{Lemma}
  
  \newtheorem{prop}[thm]{Proposition}
  \newtheorem{conj}[thm]{Conjecture}

  \theoremstyle{definition}
  \newtheorem{defin}[thm]{Definition}
  \newtheorem{bem}[thm]{Remark}
  \newtheorem{bsp}[thm]{Example}

\title{Derivatives of isotropic positive definite functions on spheres}

\author{Mara Tr\"ubner\thanks{University of Bern, Institute of Mathematical Statistics and Acturarial Science, Sidlerstrasse 5, 3012 Bern, Switzerland, e-mail: \texttt{mara.truebner@hotmail.com}} \and Johanna F.~Ziegel\thanks{University of Bern, Institute of Mathematical Statistics and Acturarial Science, Sidlerstrasse 5, 3012 Bern, Switzerland, e-mail: \texttt{johanna.ziegel@stat.unibe.ch}}}

\begin{document}

\maketitle

\begin{abstract}
We show that isotropic positive definite functions on the $d$-dimen\-sional sphere which are $2k$ times differentiable at zero have $2k+[(d-1)/2]$ continuous derivatives on $(0,\pi)$. This result is analogous to the result for radial positive definite functions on Euclidean spaces. We prove optimality of the result for all odd dimensions. The proof relies on mont\'ee, descente and turning bands operators on spheres which parallel the corresponding operators originating in the work of Matheron for radial positive definite functions on Euclidian spaces.
\end{abstract}

\section{Introduction}

Isotropic positive definite functions on spheres are important in statistics, where they occur as correlation functions of homogeneous random fields on spheres or of star shaped random particles. They are also used in approximation theory as radial basis functions for interpolating scattered data on spherical domains. \citet{Gneiting2013} recently provided a comprehensive overview of general properties and examples of parametric families of isotropic positive definite functions on spheres, and gives references to application examples.

For an integer $d \ge 1$, let $\mathbb{S}^{d}=\{x \in \R^{d+1} : \|x\|=1\}$ denote the $d$-dimensional unit sphere. A function $g: \Ss^d \times \Ss^d \rightarrow \R$ is \textit{positive definite} if
\begin{equation*}
\sum_{i=1}^n \sum_{j=1}^n c_i c_j g(x_i,x_j)\ge 0
\end{equation*}
for all $x_1,\dots,x_n \in \mathbb{S}^d$  and $c_1,\dots,c_n \in \R$. It is \emph{strictly} positive definite if equality implies $c_1 = \dots = c_n = 0$. 
The function $g: \Ss^d \times \Ss^d \rightarrow \R$ is \textit{isotropic} if there exists a function $\psi : [0,\pi] \rightarrow \R$ such that
\begin{equation}\label{sss}
g(x,y)= \psi\left(\theta(x,y)\right) \ \text{for all} \ x,y \in \Ss^d,
\end{equation}
where $\theta(x,y)=\arccos(\langle x,y \rangle)$ is the great circle or geodesic distance on $\Ss^d$, and $\langle \cdot , \cdot \rangle$ denotes the standard scalar product on $\R^{d+1}$. 

Let $\Psi_d$ and $\Psi_d^+\subseteq \Psi_d$ denote the classes of continuous functions $\psi : [0,\pi] \rightarrow \R$ with $\psi(0)=1$ such that the associated isotropic function $g$ in (\ref{sss}) is positive definite or strictly positive definite, respectively. They are nonincreasing in $d$, $\Psi_1 \supseteq \Psi_2 \supseteq \cdots \supseteq \Psi_\infty := \bigcap_{d=1}^\infty \Psi_d$, with the inclusions being strict \citep[Corollary 1]{Gneiting2013}. We set $\Psi_{\infty}^+ := \bigcap_{d=1}^{\infty}\Psi_d^+$. 

Theorem 1.2 in \citet{Ziegel2014} shows that the functions in $\Psi_d$ admit a continuous derivative of order $[(d-1)/2]$ on the open interval $(0,\pi)$. This result is analogous to the Eulidean case where it has long been known that radial positive definite functions on $\R^d$ admit $[(d-1)/2]$ continous derivatives on $(0,\infty)$ \citep{Schoenberg1938}. 

Analogously to the Euclidean case, an extension of $\psi$ to $[-\pi,\pi]$ is given by its even continuation $\psi(|\theta|)$. So by saying that $\psi$ is $2k$ times differentiable at zero, we mean that its even continuation $\psi(|\theta|)$, $\theta \in [-\pi,\pi]$, has $2k$ derivatives at $\theta=0$. Thus, if $\psi''(0)$ exists, then $\psi'(0)=0$. This convention is valid throughout the article.

\citet[Theorem 2.14]{GaspariCohn1999} demonstrate that if $k \ge 0$ and $\psi \in \Psi_d$ is $2k$ times differentiable at zero, then $\psi$ has $2k$ continuous derivatives on $(0, \pi)$. The analogous result in the Euclidean case can already be found in \citet[p.~66]{Yaglom1987}. \citet[Theorem 1]{Gneiting1999} showed the following stronger result in the Euclidean case. If $\varphi$ is a radial positive definite function on $\mathbb{R}^d$ that is $2k$ times differentiable at zero, then $\varphi$ has $2k + [(d-1)/2]$ continuous derivatives on $(0,\infty)$. In this paper, we prove the analogous result on the sphere.
\begin{thm}\label{thm:GneitingSphere}
Let $k \ge 0$ and $\psi \in \Psi_d$ for $d \ge 1$. If $\psi$ is $2k$ times differentiable at zero, then $\psi$ has $2k + [(d-1)/2]$ continuous derivatives on $(0,\pi)$.
\end{thm}
\citet{Gneiting1999} showed that his result is optimal under the given assumptions using Matheron's mont\'ee operator \citep{Matheron1965}. The same is true for Theorem \ref{thm:GneitingSphere} if $d$ is odd. As for the differentiability result of \citet[Theorem 1.2]{Ziegel2014}, optimality of Theorem \ref{thm:GneitingSphere} for even dimensions $d$ would follow immediately if a function $\psi \in \Psi_2$ with discontinuous derivative was available. 

\citet{BeatsonCastell2015} recently introduced a \emph{mont\'ee} and \emph{descente} operator on spheres in analogy to Matheron's mont\'ee and descente operators in the Euclidean case \citep{Matheron1965}. They use these operators to construct new families of strictly positive definite functions with local support on the sphere. \citet{DaleyPorcu2014} have recently studied how Matheron's mont\'ee and descente operators act on the Schoenberg measure of a radial positive definite function on $\mathbb{R}^d$. Following their approach, we study the mont\'ee and descente operator on spheres and its operation on Schoenberg sequences which are the spherical analogue to Schoenberg measures. We use the mont\'ee operator to show optimality of Theorem \ref{thm:GneitingSphere}. Complementary to the constructions of \citet{BeatsonCastell2015}, we show examples of how the descente operator can be used to obtain new families of strictly positive definite functions on spheres. 

The paper is organized as follows. In Section \ref{sec:prelim}, we introduce notation and connect the differentiability of members of $\Psi_d$ or $\Psi_{\infty}$ at zero to moments of the Schoenberg sequences. In Section \ref{sec:MD}, we discuss the mont\'ee and descente operator on spheres. The proof of Theorem \ref{thm:GneitingSphere} is given in Section \ref{sec:proof}. Some technical results are deferred to Section \ref{sec:tech}.

\section{Preliminaries}\label{sec:prelim}

The classes $\Psi_d$ can be characterized in terms of Gegenbauer expansions. The Gegenbauer polynomials $C_n^\lambda$ for $\lambda > 0$ and $n \in \N_0$ are defined by 
\[
\frac{1}{(1+r^2-2r \cos \theta)^\lambda}=\sum_{n=0}^\infty r^n C_n^\lambda(\cos\theta), \quad \theta \in [0,\pi],
\] 
where $r \in (-1,1)$. For $\lambda = 0$, we set $C_n^0(\cos \theta) = \cos(n\theta)$. \citet{Schoenberg1942} characterized the functions $\psi$ in $\Psi_d$ as having a representation
\begin{equation}\label{eq:psid}
\psi(\theta)=\sum_{n=0}^{\infty} b_{n,d} \frac{C_n^{(d-1)/2}(\cos \theta)}{C_n^{(d-1)/2}(1)} 
\end{equation}
with non-negative coefficients $b_{n,d}$ such that $\sum_{n=0}^\infty b_{n,d} = 1$. The sequence $(b_{n,d})_{n \in \mathbb{N}_0}$ is called the \emph{$d$-Schoenberg sequence} of $\psi$ following \citet{Gneiting2013}. For $d \ge 2$, the functions in $\Psi_d^+$ have $d$-Schoenberg sequences with infinitely many even and infinitely many odd coefficients that are strictly positive \citep{ChenMenegattoETAL2003}. The functions in $\Psi_1^+$ have $1$-Schoenberg sequences such that for any two integers $0\le j \le n$, there exists a $k\ge 0$ such that $b_{j+kn,1} > 0$ \citep{MenegattoOliveiraETAL2006}.

The class $\Psi_\infty$ consists of the functions $\psi$ of the form
\begin{equation}\label{eq:psiunendlich}
\psi(\theta)=\sum_{n=0}^{\infty} b_n(\cos \theta)^n\ \ \text{for} \ \ \theta \in [0,\pi],
\end{equation}
with nonnegative coefficients $b_n$ such that $\sum\nolimits_{n=0}^\infty b_n = 1$. We call the sequence $(b_n)_{n \in \mathbb{N}_0}$ the $\infty$-Schoenberg sequence of $\psi$. The functions in $\Psi_\infty^+$ have $\infty$-Schoenberg sequences with infinitely many even and infinitely many odd strictly positive coefficients \citep{Menegatto1994}.

In the case of radial positive definite functions $\varphi$ on $\R^d$, \citet[Lemma 3]{Gneiting1999} connects the existence of derivatives of (the even extension of) $\varphi$ at the origin to the existence of moments of the $d$-Schoenberg measure of $\varphi$. The following Lemma shows that a similar statement holds in the spherical case with $d$-Schoenberg measures replaced by $d$-Schoenberg sequences.

\begin{lem}\label{lemma:representation}
\begin{enumerate}
\item Let $k \ge 1$. Suppose $\psi \in \Psi_d$ with $d$-Schoenberg sequence $(b_{n,d})_{n \in \N_0}$. Then, $\psi^{(2k)}(0)$ exists if and only if $\sum_{n=0}^{\infty} b_{n,d} n^{2k}$ converges.
\item Let $k \in \{1,2\}$ and $\psi \in \Psi_\infty$ with $\infty$-Schoenberg sequence $(b_n)_{n \in \N_0}$. Then $\psi^{(2k)}(0)$ exists if and only if $\sum_{n=0}^{\infty} b_{n,d} n^{k}$ converges.
\end{enumerate}
\end{lem}

\begin{proof}
Let $\psi \in \Psi_d \subseteq \Psi_1$. By \citet[p.~120]{Yaglom1987} the $2\pi$-periodic extension $\varphi_1$ of the function $\psi$ is a radial positive definite function on $\R$. Its Schoenberg representation, as given in for example in \citet[Eq.~1]{Gneiting1999}, is
\begin{equation*}
\varphi_1(\theta) = \int_{[0,\infty)} \cos(u \theta) G_1(\d u) = \sum_{n=0}^\infty b_{n,1} \cos(n \theta) = \psi(\theta),
\end{equation*}
where $G_1$ is a probability measure on $[0,\infty)$.
As $\varphi_1$ is $2\pi$-periodic, $G_1$ has to be of the form $G_1=\sum_{n=0}^\infty b_{n,1} \delta_n$ with $b_{n,1} \ge 0$ and $\sum_{n=0}^\infty b_{n,1}=1$, where $\delta_x$ is the dirac measure at $x$. 
By \citet[Lemma 3]{Gneiting1999}, we have that $\varphi_1^{(2k)}(0)$ exists if and only if $\int_{[0,\infty)} u^{2k} G_1(\d u)=\sum_{n=0}^\infty n^{2k} b_{n,1}< \infty$. Using the formulae in Proposition \ref{uglyformula}(a), we obtain
\begin{align*}
&\sum_{n=0}^\infty  n^{2k} b_{n,1} = \sum_{n=1}^\infty (2n)^{2k}\, b_{2n,1} + \sum_{n=1}^\infty(2n-1)^{2k} \,b_{2n-1,1}
\\ &= \sum_{n=1}^\infty (2n)^{2k} \,2 \sum_{j=n}^\infty b_{2j,d} \kappa_d(2j,2n) + \sum_{n=1}^\infty (2n-1)^{2k} \,2 \sum_{j=n}^\infty b_{2j-1,d} \kappa_d(2j-1,2n-1)
\\ &= \sum_{j=1}^\infty b_{2j,d} \,2 \sum_{n=1}^j (2n)^{2k} \kappa_d(2j,2n) + \sum_{j=1}^\infty b_{2j-1,d}\, 2 \sum_{n=1}^j (2n-1)^{2k} \kappa_d(2j-1,2n-1).
\end{align*}
Lemma \ref{lem:moments} shows that $\sum_{n=0}^\infty n^{2k} b_{n,1} < \infty$ if and only if $\sum_{n=0}^\infty n^{2k} b_{n,d} < \infty$.

The proof of part (b) follows the lines of the proof of part (a) where we use the formulae in Proposition \ref{uglyformula}(b) instead of Proposition \ref{uglyformula}(a) and Lemma \ref{lem:inftymoments} instead of Lemma \ref{lem:moments}. 
\end{proof}

In contrast to the Euclidean case it is not immediate to give a formula for $\psi^{(2k)}(0)$ in terms of the $d$-Schoenberg sequence for all $k$. Differentiating \eqref{eq:psid} yields
\begin{equation}\label{eq:psi2}
\psi''(0) = -\sum_{n=0}^\infty b_{n+1,d}\frac{(n+1)(n+d)}{d}
\end{equation}
if one of the conditions in Lemma \ref{lemma:representation}(a) for $k=1$ is fulfilled. For $k=2$, we obtain
\[
\psi^{(4)}(0) = \sum_{n=0}^\infty b_{n+1,d}\frac{(n+1)(n+d)}{d}\Big(1 + \frac{3n(n+d+1)}{d+2}\Big).
\]
If $\psi \in \Psi_{\infty}$, we obtain by differentiating \eqref{eq:psiunendlich}
\begin{equation*}
\psi''(0) =  -\sum_{n=0}^\infty b_{n+1,d}(n+1), \quad \psi^{(4)}(0) =  \sum_{n=0}^\infty b_{n+1,d}(n+1)(3n+1)
\end{equation*}
if one of the conditions in Lemma \ref{lemma:representation}(b) is fulfilled for $k=1$ or $k=2$, respectively. We conjecture that the second part of Lemma \ref{lemma:representation} holds for all $k \ge 1$. In order to obtain the result, one would have to demonstrate the following conjecture, which we have only been able to check using Mathematica computer algebra for $k \in \{3,\dots,15\}$. 
\begin{conj}
For $k \ge 1$, there is a constant $c(k) > 0$ such that, as $j \to \infty$,
\begin{align*}
2^{-2j+1}\sum_{n=1}^j(2n)^{2k} \binom{2j}{j+n} \sim c(k) j^k, \quad 2^{2j}\sum_{n=1}^j(2n-1)^{2k} \binom{2j-1}{j+n-1} \sim c(k) j^k.
\end{align*}
\end{conj}

\begin{bem}
Let $\Psi_d^c$ denote the class of the functions $\psi \in \Psi_d$ with $\psi(\theta)=0$ for $ \theta \ge c$, where $c \in (0,\pi]$ and such that $\psi''(0)$ exists. It is an  open  question  to find
\begin{equation*}
a_d^c = \inf _{\psi \in \Psi_d^c} [-\psi''(0)],
\end{equation*}
see \citet[Appendix C: Problem 3]{Gneiting2013}. For $d=1$ and $d=3$ it is shown in \citet{Gneiting2013} that
\begin{equation}\label{ineq:inf} 
a_d^c \le \frac{1}{c^2} \frac{4}{d} j_{(d-2)/2}^2, 
\end{equation}
where $j_\alpha$ denotes the first positive zero of the Bessel function $J_\alpha$. Formula \eqref{eq:psi2} yields the solution for $c = \pi$. Finding $a_d^\pi$ is equivalent to find the minimum of $\sum_{n=0}^\infty b_{n,d} c_n$ with  $b_{n,d} \ge 0$, and $c_n= n(n-1+d)/d$ under the conditions $\sum_{n=0}^\infty b_{n,d}=1 $ and $\sum_{n=0}^\infty (-1)^n b_{n,d} =0$. The minimum is attained if $b_{0,d} = b_{1,d} = 1/2$ and $b_{n,d} = 0$ for $n\ge 2$, thus $a_d^\pi = 1/2$. For $d=1$ the upper bound in \eqref{ineq:inf} equals $1$ and for $d=3$ it is equal to $4/3$.
\end{bem}

\section{The mont\'ee and descente operator on spheres}\label{sec:MD}

Similar to \citet{BeatsonCastell2015}, we make the following definition.

\begin{defin} Let $\psi: [0,\pi] \to \R$.

\begin{enumerate}
\item If $\int_0^\pi \sin(\beta) \psi(\beta) \d \beta$ is finite and non-zero, then the operator $I_{S}$, called \textit{mont\'ee on spheres}, is defined as
\[I_{S}\psi(\theta) := \frac{\int_\theta^\pi \sin(\beta) \psi(\beta) \d \beta}{\int_0^\pi \sin(\beta) \psi(\beta) \d \beta}, \quad \theta \in [0,\pi]. \]
\item If $\psi$ is differentiable and has a finite second derivative $\psi''(0) \neq 0$, then the operator $D_{S}$, called \textit{descente on spheres}, is defined as
\[D_{S}\psi(\theta) := \frac{\psi'(\theta)}{\sin(\theta) \psi''(0)}, \quad \theta \in (0,\pi).\]
\end{enumerate}
\end{defin}
The mont\'ee and descente operator on spheres are inverse to each other under the following conditions which can be checked easily. 

\begin{lem} Let $\psi: [0,\pi] \to \R$.
\begin{enumerate}
\item If $\psi$ is continuously differentiable on $(0,\pi)$, $\psi(\pi)=0$, $\psi(0)=1$ and $ |\psi''(0)| \in (0,\infty)$ then $I_S$ inverts $D_S \psi$ to $\psi$.
\item If $\int_0^\pi sin \beta \psi(\beta) \d \beta$ is finite and non-zero and $\psi(0)=1$, then $D_S(I_S \psi)= \psi$.
\end{enumerate}
\end{lem}

For functions $\psi \in \Psi_d$, the next propositions characterize when $I_{S}\psi$ and $D_{S}\psi$ belong to the function classes $\Psi_{d'}$ for $d' = d-2$ or $d' = d+2$, respectively.

\begin{prop}\label{prop:montee}
Let $\psi \in \Psi_d$ for $d\ge 3$ with $d$-Schoenberg sequence $(b_{n,d})_{n \in \mathbb{N}_0}$.
\begin{enumerate}
\item The mont\'ee $I_S\psi$ of $\psi$ is well-defined and $I_S\psi \in \Psi_{d-2}$, if and only if 
\begin{equation}\label{eq:condcd}
c(d) := \sum_{n=0}^\infty b_{n,d} \frac{(-1)^{n}(d-2)}{(n+1)(n+d-2)} \ge 0.
\end{equation}
In this case, $I_{S}\psi \in \Psi_{d-2}$ and has $(d-2)$-Schoenberg sequence $(a_{n,d-2})_{n \in \mathbb{N}_0}$ with 
$a_{0,d-2} = c(d)/G_1(\mathbb{N}_0)$, $a_{n,d-2} = (d-2)b_{n-1,d}/(n(n+d-3)G_1(\N_0))$ for $n \ge 1$,  where 
$G_1(\N_0) :=  2(d-2)\sum_{n=0}^\infty b_{2n,d}/((2n+1)(2n+d-2))$.
\item The condition \eqref{eq:condcd} is equivalent to $\int_0^\pi f_d(\theta)\psi(\theta)\d\theta \ge 0$, where
\begin{align*}
f_d(\theta) &= \begin{cases}\frac{1}{\pi}\theta \sin\theta - \cos\theta\sum_{\ell=1}^{(d-3)/2}(\sin\theta)^{2\ell}\Gamma(\ell)^2 2^{2\ell-2}/(\pi\Gamma(2\ell)), &\text{$d \ge 3$ odd,}\\
\frac{1}{2}\sin\theta  - \cos\theta \sum_{\ell=1}^{d/2 - 1} (\sin\theta)^{2\ell - 1}\Gamma\big(\frac{2\ell-1}{2}\big)^2 2^{2\ell-3}/(\pi\Gamma(2\ell-1)),&\text{$d\ge 4$ even,}\end{cases}\\
&= \sin\theta\mathbbm{1}\{\theta > \pi/2\} + \frac{2^{d-3}\Gamma\big(\frac{d-1}{2}\big)^2}{\pi\Gamma(d-1)}\cos\theta (\sin\theta)^{d-1} \,_2F_1\Big(1,\frac{d-1}{2},\frac{d}{2},(\sin\theta)^2\Big).
\end{align*}
\item A sufficient condition for \eqref{eq:condcd} to hold is that $\psi \ge 0$. 
\end{enumerate}
\end{prop}

\begin{proof}
For $d\ge 4$, using that $C_n^{\lambda}(1) = \Gamma(n+2\lambda)/(n!\Gamma(2\lambda))$ \citep[18.6.1]{dlmf} and \citet[18.9.19]{dlmf}, we obtain with Lebesgue's dominated convergence theorem that
\begin{align*}
\int_{\theta}^\pi \sin \beta \psi(\beta) \d \beta &= \sum_{n=0}^\infty \frac{b_{n,d}}{C_n^{(d-1)/2}(1)} \frac{1}{(d-3)} \big(C_{n+1}^{(d-3)/2}(\cos \theta)-(-1)^{(n+1)}C_{n+1}^{(d-3)/2}(1)\big) \\ 
&= \sum_{n=0}^{\infty} \frac{b_{n,d}(-1)^{n}(d-2)}{(n+1)(n+d-2)} + \sum_{n=0}^\infty  \frac{b_{n,d}(d-2)}{(n+1)(n+d-2)} \frac{C_{n+1}^{(d-3)/2}(\cos \theta)}{C_{n+1}^{(d-3)/2}(1)}\\
&= G_1(\N_0) \sum_{n=0}^{\infty} a_{n,d-2} \frac{C_{n}^{(d-3)/2}(\cos \theta)}{C_{n}^{(d-3)/2}(1)}
\end{align*}
with $a_{n,d-2}$ as indicated. The assumptions guarantee that $a_{n,d-2} \ge 0$ for all $n \ge 0$, and it is easily checked that they sum to one because $G_1(\N_0) = \int_0^\pi \sin \beta \psi(\beta)\d \beta$. This exhibits $I_{S}\psi$ as an element of $\Psi_{d-2}$. The case $d=3$ is shown analogously. 

To show part (b), we proceed inductively. For $d \ge 5$, we have by \citet[Corollary 3]{Gneiting2013}
\begin{align*}
c(d) &= \sum_{n=0}^{\infty}(-1)^n \Big(\frac{n+d-3}{(n+1)(2n + d- 3)}b_{n,d-2} - \frac{n+2}{(n+d-2)(2n+d+1)}b_{n+2,d-2}\Big)\\
&= \sum_{n=0}^{\infty} (-1)^n b_{n,d-2}\frac{d-4}{(n+1)(n+d-4)} - b_{1,d-2}\frac{1}{(d-3)(d-1)}\\
&= c(d-2) - b_{1,d-2}\frac{1}{(d-3)(d-1)},
\end{align*}
hence
\begin{equation}\label{eq:rec}
c(d) = \begin{cases} 
c(3) - \sum_{\ell=1}^{(d-3)/2}b_{1,2\ell+1}/(4\ell (\ell+1)),& \text{$d\ge 5$ odd,}\\
c(4) - \sum_{\ell = 2}^{d/2 - 1}b_{1,2\ell}/((2\ell - 1)(2\ell + 1)),& \text{$d\ge 6$ even.}\end{cases}
\end{equation}
As $C_1^{(d-1)/2}(\cos\theta) = (d-1)\cos(\theta)$, we obtain for $d \ge 2$,
\begin{equation}\label{eq:b1d}
b_{1,d} = \frac{(d+1)(d-1)}{2^{3-d}\pi}\frac{\Big(\Gamma\big(\frac{d-1}{2}\big)\Big)^2}{\Gamma(d-1)} \int_0^\pi \cos\theta (\sin\theta)^{d-1}\psi(\theta)\d\theta.
\end{equation}
Using \citet[Corollary 3]{Gneiting2013} we obtain for $d=3$
\begin{align*}
c(3)&= \sum_{n=0}^\infty b_{n,3} \frac{(-1)^{n}}{(n+1)^2} = b_{0,1} - \frac{1}{2}b_{2,1} + \frac{1}{2}\sum_{n=1}^{\infty} \frac{(-1)^{n}}{n+1}(b_{n,1} - b_{n+2,1})\\ 
&= \frac{1}{\pi}\lim_{N\to\infty}\sum_{n=0}^{N} \frac{(-1)^{n}}{n+1}\int_0^\pi (\cos(n\theta) - \cos((n+2)\theta))\psi(\theta)\d\theta\\ 
&= \frac{1}{\pi}\int_0^\pi \psi(\theta)\d\theta - \frac{1}{2\pi}\int_0^\pi \cos\theta\,\psi(\theta) \d\theta \\&\quad+ \frac{1}{\pi}\lim_{N\to\infty}\sum_{n=2}^{N}(-1)^n\Big(\frac{1}{n+1} - \frac{1}{n-1}\Big)\int_0^\pi \cos(n\theta)\psi(\theta)\d\theta \\ 
&= \frac{1}{\pi}\int_0^\pi \psi(\theta)\d\theta - \frac{1}{2\pi}\int_0^\pi \cos\theta\,\psi(\theta) \d\theta  - \frac{2}{\pi}\int_0^\pi \sum_{n=2}^{\infty}\frac{(-1)^n}{n^2 - 1}\cos(n\theta)\psi(\theta)\d\theta,
\end{align*}
where the last step follows by dominated convergence. We have for $\theta \in (-\pi,\pi)$,
$\sum_{n=2}^{\infty}(-1)^n\cos(n\theta)/(n^2 - 1) =1/2- (1/4)\cos\theta - \theta\sin\theta$,
and hence
\begin{equation}\label{eq:bnd3}
c(3)= \sum_{n=0}^\infty b_{n,3} \frac{(-1)^{n}}{(n+1)^2} = \frac{1}{\pi}\int_0^\pi \theta \sin\theta\,\psi(\theta)\d\theta.
\end{equation}
For $d=4$, we obtain
\begin{align}
c(4) &= \sum_{n=0}^{\infty} b_{n,4} \frac{2(-1)^n}{(n+1)(n+2)} = \sum_{n=0}^\infty (-1)^n\Big(\frac{1}{2n + 1}b_{n,2} - \frac{1}{2n + 5} b_{n+2,2}\Big)\nonumber\\
&= b_{0,2} - \frac{1}{3}b_{1,2} = \frac{1}{2}\int_0^\pi (1 - \cos\theta)\sin\theta\psi(\theta)\d\theta\label{eq:bnd4}.
\end{align}
Combining \eqref{eq:rec}, \eqref{eq:b1d} and \eqref{eq:bnd3}, we obtain the claim for any odd $d \ge 3$, and for any even $d \ge 4$, it follows from \eqref{eq:rec}, \eqref{eq:b1d} and \eqref{eq:bnd4}.
Part (c) follows from \citet[Theorem 2.2]{BeatsonCastell2015}.
\end{proof}

\begin{prop}\label{prop:descente}
Let $\psi \in \Psi_d$ for $d\ge 1$ with $d$-Schoenberg sequence $(b_{n,d})_{n \in \mathbb{N}_0}$. Suppose that $\psi \not=1$. 
\begin{enumerate}
\item The descente $D_{S}\psi$ is well defined and $D_{S}\psi \in \Psi_{d+2}$, if and only if $G_2(\N_0) := \sum_{n=0}^\infty b_{n+1,d}(n+1)(n+d)  < \infty$. In this case it has $(d+2)$-Schoenberg seqence $(c_{n,d+2})_{n \in \N_0}$ with $c_{n,d+2} = b_{n+1,d} (n+1)(n+d)/G_2(\N_0).$
\item The condition $G_2(\N_0) < \infty$ in (a) is equivalent to $|\psi''(0)| < \infty$.
\end{enumerate}
\end{prop}
 
\begin{proof}
Using the assumption $G_2(\N_0) < \infty$ and \citet[18.9.19]{dlmf}, we obtain
\[
\psi'(\theta) = -\sin \theta\sum_{n=0}^\infty b_{n+1,d}\frac{(n+1)(n+d)}{d} \frac{C_n^{(d+1)/2}(\cos\theta)}{C_n^{(d+1)/2}(1)},
\]
which implies $\psi'(0) = \psi'(\pi)=0$ and $\psi''(0) = -G_2(\N_0)$ by \eqref{eq:psi2}, hence the descente $D_S\psi$ is well-defined has the claimed $(d+2)$-Schoenberg sequence. Part (b) is a consequence of Lemma \ref{lemma:representation}.
\end{proof}
 

\begin{prop}\label{prop:MDinfty} Let $\psi \in \Psi_\infty$ with $\infty$-Schoen\-berg sequence $(b_n)_{n \in \N_0}$.
\begin{enumerate}
\item The mont\'ee $I_S\psi$ is well defined and in $\Psi_\infty$, if and only if, 
\begin{equation}\label{eq:condinfty}
\sum_{n=0}^\infty  b_n\frac{(-1)^n}{n+1} \ge 0.
\end{equation}
It has $\infty$-Schoenberg sequence $(a_n)_{n \in \N_0}$ with $a_n = b_{n-1}/(nG_1(\N_0))$ for $n \ge 1$ and $a_0= (1/G_1(\N_0))\sum_{n=0}^\infty (-1)^nb_n /(n+1)$, where $G_1(\N_0) := 2\sum_{n=0}^\infty b_{2n} /(2n+1)$.
\item The condition at \eqref{eq:condinfty} is equivalent to $\int_{\pi/2}^\pi \sin \theta \psi(\theta) \d \theta \ge 0$.
\item Let $\psi \neq 1$.  The descente $D_S\psi$ is well defined and in $\Psi_\infty$, if and only if, $G_2(\N_0) := \sum_{n=0}^\infty b_{n+1}(n+1) < \infty$. It has $\infty$-Schoenberg sequence $(c_n)_{n \in \N_0}$ with $c_n = b_{n+1} (n+1)/ G_2(\N_0)$. 
\item The condition $G_2(\N_0) < \infty$ in $(c)$ is equivalent to $|\psi''(0)| < \infty$.
\end{enumerate}
\end{prop}

\begin{proof}
The proofs of parts (a) and (c) are analogous to the proofs of Proposition \ref{prop:montee}(a) and Proposition \ref{prop:descente}(a), respectively. Part (b) follows because $\int_{\pi/2}^{\infty}\sin\theta\psi(\theta)\d\theta$ $= \sum_{n=0}^{\infty}b_n(-1)^n/(n+1)$, and part (d) is a consequence of Lemma \ref{lemma:representation}. 
\end{proof}

\begin{bem}
The formulae for the Schoenberg sequences under the mont\'ee and descente operators immediately imply that they preserve strict positive definiteness, in the sense that if $\psi \in \Psi_d^+$, then $D_S\psi \in \Psi_{d+2}^+$, and so forth, with one exception: If $\psi \in \Psi_3^+$, it may happen that $I_S\psi \in \Psi_1\backslash \Psi_1^+$ as for members of $\Psi_1$ it is only a necessary but not a sufficient condition for strict positive definiteness that infinitely many even and infinitely many odd coefficients of the Schoenberg sequence are strictly positive \citep{MenegattoOliveiraETAL2006}.
\end{bem}

In the following examples, we apply the descente on spheres to parametric families of positive definite functions. \citet{BeatsonCastell2015} have applied the mont\'ee on spheres to parallel the construction of \citet{Wendland1995} on the sphere. 

\begin{bsp}
The Multiquadric family
\[ \psi_{\tau,\delta}(\theta)= \frac{(1-\delta)^{2\tau}}{(1+ \delta^2 -2 \delta \cos \theta)^\tau} \]
belongs to the class $\Psi_{\infty}^+$ for $ \tau >0$ and $\delta \in (0,1)$ \citep{Gneiting2013}. One can check that the descente is applicable using the criterion in Proposition \ref{prop:MDinfty}(d). We obtain 
\[
D_S \psi_{\tau,\delta}(\theta)= \frac{(1-\delta)^{2(\tau + 1)}}{(1 + \delta^2 -2\delta \cos \theta)^{\tau+1}} = \psi_{\tau+1,\delta}(\theta),
\]
hence the Multiquadric family is closed under application of the descente. For $\tau > 0$, we have $\psi_{\tau,\delta} \ge 0$. Applying the mont\'ee operator for $\tau\not=1$ yields
\[
I_S \psi_{\tau,\delta}(\theta)= \frac{\psi_{\tau-1,\delta}(\theta) - \big(\frac{1-\delta}{1+\delta}\big)^{2(\tau-1)}}{1- \big(\frac{1-\delta}{1+\delta}\big)^{2(\tau-1)}}\in \Psi_{\infty}^+.
\]
For $\tau > 1$, the above expression is just a translated and rescaled version of $\psi_{\tau-1,\delta}$ such that it vanishes at $\theta = \pi$. For $\tau \in (0,1)$, the denomiator is negative, therefore, $-\psi_{\tau-1,\delta}$ is positive definite on spheres of all dimensions. Finally, for $\tau = 1$,
\[
I_S \psi_{1,\delta} = \frac{2\log(1+\delta) - \log(1+\delta^2 - 2\delta\cos\theta)}{2\log(1+\delta) - 2\log(1-\delta)} \in \Psi_{\infty}^+.
\]
\end{bsp}

Given a radial positive definite function $\varphi$ on $\mathbb{R}^3$, by the construction of \citet{Yadrenko1983}, one obtains a function $\psi \in \Psi_2^+$ by setting
\begin{equation}\label{eq:Yadrenko}
\psi_1(\theta) = \varphi\Big(2\sin\frac{\theta}{2}\Big), \quad \theta \in [0,\pi].
\end{equation}
Alternatively, if $\varphi(t) = 0$ for $t \ge \pi$, \citet[Theorem 3]{Gneiting2013} shows that $\psi_2 = \varphi|_{[0,\pi]} \in \Psi_3^+$. Suppose that $\varphi''(0)$ exists and does not vanish. Then, $\psi_1''(0) = \psi_2''(0) = \varphi''(0)$. We can apply the descente operator to both, $\psi_1$ and $\psi_2$, and obtain
\begin{align*}
D_S \psi_1(\theta) &= \frac{1}{\varphi''(0)2\sin\frac{\theta}{2}}\varphi'\Big(2\sin\frac{\theta}{2}\Big)\in \Psi_4^+,\\ 
\quad D_S\psi_2(\theta) &= \frac{1}{\varphi''(0)\sin\theta}\varphi'(\theta) \in \Psi_5^+.
\end{align*}

\begin{bsp}
The Gaspari and Cohn (\citeyear{GaspariCohn1999}) function $\varphi_{GC}$  is defined by
\[
\varphi_{GC}(t) = \begin{cases} 1 - \frac{20}{3}t^2 + 5t^3 + 8t^4 - 8t^5, & t \in [0,\frac{1}{2}],\\
\frac{1}{3}t^{-1}(8t^2 + 8t - 1)(1-t)^4, & t \in [\frac{1}{2},1],\\
0, & t \in [1,\infty).
\end{cases}
\]
For all $c > 0$, the function $\varphi_c(t) = \varphi_{GC}(t/c)$ is a radial positive definite function on $\R^3$. We have $\varphi_c''(0) = -40/(3c^2)$. Applying the construction at \eqref{eq:Yadrenko} to $\varphi_{GC}(\cdot/(2\sin(c/2)))$ with $c \in (0,\pi]$ yields a compactly supported member $\psi_{1,c}$ of $\Psi_2^+$ with support $[0,c]$; see \citet[Eq.~30]{Gneiting2013}. Applying the descente to $\psi_{1,c}$ and $\psi_{2,c} = \varphi_c|_{[0,\pi]}$ yields
\[
D\psi_{1,c}(\theta) = \tilde{\varphi}\Big(\frac{\sin\frac{\theta}{2}}{\sin\frac{c}{2}}\Big) \in \Psi_4^+, \quad D_S\psi_{2,c}(\theta) = \frac{\theta}{\sin\theta}\tilde{\varphi}\Big(\frac{\theta}{c}\Big)\in \Psi_5^+,
\]
where
\[
\tilde{\varphi}(t) = \begin{cases}1 - \frac{9}{8}t -\frac{12}{5}t^2 +3 t^3, & \theta \in [0,\frac{1}{2}],\\
\frac{1}{40} \big(-1 -3t + 24 t^2 + 40t^3\big)\big(\frac{1}{t}-1\big)^3, & \theta \in [\frac{1}{2},1],\\
0, & \theta \in [1,\infty).\end{cases}
\]

\end{bsp}

\begin{bsp}
The $C_2$-Wendland function \citep{Wendland1995} given by
\[
\psi(\theta) = \Big(1 + \tau\frac{\theta}{c}\Big)\Big(1-\frac{\theta}{c}\Big)_+^\tau, \quad \theta \in [0,\pi],
\]
is in $\Psi_3^+$ for $c \in (0,\pi)$ and $\tau \ge 4$ \citep{Gneiting2013}. One can check that $\psi''(0) = -\tau(1+\tau)/c^2$ and 
\[
D\psi(\theta) = \frac{\theta}{\sin\theta}\Big(1-\frac{\theta}{c}\Big)_+^{\tau-1} \in \Psi_5^+.
\]
For the $C_4$-Wendland function, given by
\[
\psi(\theta) = \Big(1 + \tau\frac{\theta}{c}+\frac{\tau^2 - 1}{3}\frac{\theta^2}{c^2}\Big)\Big(1-\frac{\theta}{c}\Big)_+^\tau, \quad \theta \in [0,\pi],
\]
which is in $\Psi_3^+$ for $c \in (0,\pi)$, $\tau \ge 6$, see \citet{Gneiting2013}, we obtain
\[
D\psi(\theta) = \frac{\theta}{\sin\theta}\Big(1 + (\tau-1)\frac{\theta}{c}\Big)\Big(1-\frac{\theta}{c}\Big)_+^{\tau-1} \in \Psi_5^+.
\]
\end{bsp}

To conclude this section we show the following result on the differentiability of the mont\'ee. It allows to demonstrate optimality of Theorem \ref{thm:GneitingSphere}; see Section \ref{sec:proof}.

\begin{lem}\label{lem:diffmontee}
Let $k \ge 0$ and $d\ge 3$. If $\psi \in \Psi_d$ fulfills the condition of Proposition \ref{prop:montee}(b) and is $2k$ times differentiable at zero, then $I_S\psi \in \Psi_{d-2}$ is $2k+2$ times differentiable at zero.
\end{lem}

Inductively, we can express $(I_S\psi)^{(j)}$ for $j \ge 1$ in terms of the derivatives of $\psi$.
\begin{multline}\label{eq:Ider}
(I_S\psi)^{(j)}(\theta) = \frac{-1}{\int_0^\pi \sin \beta \psi(\beta) \d \beta}\Big(\sin\theta\sum_{\substack{\ell=0\\\text{$\ell$ even}}}^{j-1}\binom{j-1}{\ell}(-1)^{[\ell/2]}\psi^{(j-1-\ell)}(\theta)\\
 + \cos\theta \sum_{\substack{\ell=0\\\text{$\ell$ odd}}}^{j-1}\binom{j-1}{\ell}(-1)^{[\ell/2]}\psi^{(j-1-\ell)}(\theta)\Big).
\end{multline}

\begin{proof}[Proof of Lemma \ref{lem:diffmontee}]
If $\psi$ is $2k$ times differentiable at zero, it has $2k$ continuous derivatives on $(0,\pi)$ by \citet[Theorem 2.14]{GaspariCohn1999}. By \eqref{eq:Ider}, $I_S\psi$ is then $(2k+1)$ times differentiable on $(0,\pi)$ and $I_S\psi(|\theta |)$ is $(2k+1)$ times differentiable at zero. Since odd order derivatives of $I_S\psi(|\theta |)$ are odd functions of a real argument, all the odd order derivatives of $I_S\psi(|\theta |)$ vanish at $\theta =0$. Thus, using \eqref{eq:Ider} and l'H\^opital's rule,
\begin{align*}
(I_S&\psi)^{(2k+2)}(0) =\lim_{\theta \to 0} \frac{(I_S \psi)^{(2k+1)}(\theta)}{\sin \theta}\\
&= \frac{-1}{\int_0^\pi \sin \beta \psi(\beta) \d \beta}\Big((-1)^k\psi(0) + \sum_{\ell=0}^{k-1}(-1)^{\ell}\Big(\binom{2k}{2\ell}+\binom{2k}{2\ell+1}\Big)\psi^{(2k-2\ell)}(0)\Big)\\
&= \frac{-1}{\int_0^\pi \sin \beta \psi(\beta) \d \beta}\sum_{\ell=0}^{k}(-1)^{\ell}\binom{2k+1}{2\ell+1}\psi^{(2k-2\ell)}(0).
\end{align*}
\end{proof}

\section{Proof of Theorem \ref{thm:GneitingSphere}}\label{sec:proof}
In \citet{Ziegel2014}, a \emph{turning bands operator} for isotropic positive definite functions on spheres is introduced. It is analogous to the Euclidean case, where the turning bands operator originates in the work of \citet{Matheron1972}.
For a sequence $\beta = (\beta_n)_{n \in \mathbb{N}_0}$ of real numbers and an integer $k$, we define the sequence $\beta\circ\tau_{k}$ as follows. If $k > 0$ we set $(\beta\circ\tau_{k})_n = 0$ for $n < k$ and $(\beta\circ\tau_{k})_n = \beta_{n-k}$ for $n \ge k$. If $k \le 0$ we put $(\beta\circ\tau_{k})_n = \beta_{n-k}$ for all $n \ge 0$. For a summable sequence $\beta=(\beta_n)_{n \in \mathbb{N}}$ of nonnegative numbers $\beta_n$ we define $\psi_d(\beta,\theta)$ for $\theta \in [0,\pi]$ as 
\[
\psi_d(\beta,\theta) = \sum_{n=0}^{\infty} \beta_n \frac{C_n^{(d-1)/2}(\cos\theta)}{C_n^{(d-1)/2}(1)}. 
\]
\citet[Proposition 4.4]{Ziegel2014} shows that, for all $\theta \in [0,\pi]$,
\begin{equation}\label{eq:turnop}
\psi_d(\beta,\theta) = \beta_0 + \cos \theta \ \psi_{d+2}(\beta \circ \tau_{-1},\theta) + \frac{1}{d}\sin \theta \ \psi_{d+2}'(\beta \circ \tau_{-1},\theta),
\end{equation}
and
\begin{equation}\label{eq:invop}
\frac{1}{d}(\sin \theta)^d \ \psi_{d+2}(\beta \circ \tau_{-1},\theta)= \int_0^\theta(\sin r)^{d-1} (\psi_d(\beta,r)-\beta_0) \d r.
\end{equation}

\begin{bem}
The turning bands operator at \eqref{eq:turnop} shows that there is a one-to-one correspondence between the members of $\Psi_{d+2}$ and the members $\psi$ of $\Psi_d$ with $b_{0,d} = 0$, that is, $\int_0^\pi (\sin\theta)^{d-1}\psi(\theta) = 0$.
In particular, if $\psi_{d+2} \in \Psi_{d+2}$ is compactly supported, then so is the corresponding $\psi_d \in \Psi_d$ associated via \eqref{eq:turnop} with $b_{0,d} = 0$. While we are not making use of these facts for the results of this paper, it may be an interesting way to characterize the set $\Psi_{d}\backslash \Psi_{d+2}$ in terms of Schoenberg sequences.
\end{bem}

\begin{proof}[Proof of Theorem \ref{thm:GneitingSphere}]
Suppose that $\psi \in \Psi_d$ is $2k$ times differentiable at zero with $d$-Schoenberg seqence $b = (b_{n,d})_{n \in \mathbb{N}_0}$. Applying the turning bands operator at \eqref{eq:turnop} $[(d-1)/2]$ times yields an element $\bar{\psi} \in \Psi_1$ or $\bar{\psi} \in \Psi_2$ with $d$-Schoenberg sequence $\bar{b} = b \circ \tau_{[(d-1)/2]}$. 
By Lemma \ref{lemma:representation}, $\bar{\psi}$ is $2k$ times differentiable at zero. By \citet[Theorem 2.14]{GaspariCohn1999}, $\bar{\psi}$ has $2k$ continuous derivatives on $(0,\pi)$. By \eqref{eq:turnop}, $\psi$ has $[(d-1)/2]$ more continuous derivatives than $\bar{\psi}$ on $(0,\pi)$.
\end{proof}
To show optimality of Theorem \ref{thm:GneitingSphere} for odd $d \ge 1$, we provide an element of $\Psi_d$ which is $2k$ times differentiable at zero but whose derivative of order $2k + [(d-1)/2] + 1$ does not exist on $(0,\pi)$.

Let $c \in (0,\pi)$. Then the function $\psi(\theta) = (1-\theta/c)_+$ belongs to $\Psi_1$. Its first derivative does not exist at $\theta = c$. Let $b$ be the $1$-Schoenberg sequence of $\psi$. Applying \eqref{eq:invop} $(d'-1)/2$ times, we find that the derivative of order $1 + (d'-1)/2$ of $\psi_{d'}(b \circ \tau_{-(d'-1)/2},\cdot)$ does not exist at the point $\theta=c$ and
\[
\bar{\psi} := \frac{\psi_{d'}(b \circ \tau_{-(d'-1)/2},\cdot) + C}{\psi_{d'}(b \circ \tau_{-(d'-1)/2},0) + C} \in \Psi_{d'},
\]
where $C$ is a constant such that $\psi_{d'}(b \circ \tau_{-(d'-1)/2},\cdot) + C \ge 0$ on $[0,\pi]$. Then $(I_S)^k\bar{\psi} \in \Psi_{d'-2k}$ is $2k$ times differentiable at zero by Lemma \ref{lem:diffmontee}. By \eqref{eq:Ider}, it follows that the derivative of order $1+(d'-1)/2 + k$ of $(I_S)^k\bar{\psi}$ does not exist at $\theta = c$. Choosing $d + 2k = d'$, we obtain a function with the desired property.

\section{Technical results}\label{sec:tech}

\begin{prop}\label{uglyformula}
Let $\psi \in \Psi_d \subseteq \Psi_1$ with Schoenberg sequences $(b_{n,d})_{n\in\mathbb{N}_0}$ and $(b_{n,1})_{n \in \mathbb{N}_0}$. Then, for $d \ge 2$, we obtain $b_{0,1} = \sum_{j=0}^\infty b_{2j,d} \kappa_d(2j,0)$ and for $n \ge 1$
\[
b_{2n,1} = 2 \sum_{j=n}^\infty b_{2j,d} \kappa_d(2j,2n),\quad b_{2n-1,1}= 2 \sum_{j=n}^\infty b_{2j-1,d} \kappa_d (2j-1, 2n-1)
\]
with
\[\kappa_d(j,n) := \frac{\Gamma(d-1)}{\Gamma\big(\frac{d-1}{2}\big)^2}\frac{\Gamma\big(\frac{d-1+j-n}{2}\big)\Gamma\big(\frac{d-1+j+n}{2}\big)j! }{\Gamma\big(\frac{j-n + 2}{2}\big)\Gamma\big(\frac{j+n + 2}{2}\big)\Gamma(d-1 + j)}.
\]
If $\psi \in \Psi_{\infty} \subseteq \Psi_1$ with Schoenberg sequences $(b_n)_{n \in \mathbb{N}_0}$ and $(b_{n,1})_{n \in \mathbb{N}_0}$, then $b_{0,1} = \sum_{j=0}^\infty b_{2j} \tau(2j,0)$ and for $n \ge 1$, we have
\[
b_{2n,1} = 2\sum_{j=n}^\infty b_{2j} \tau(2j,2n),\quad b_{2n-1,1} = 2\sum_{j=n}^\infty b_{2j-1} \tau (2j-1,2n-1), 
\]
where $\tau(j,n) := 2^{-j} j!/(\Gamma\big((j-n + 2)/2\big)\Gamma\big((j+n + 2)/2\big))$.
\end{prop}

\begin{proof}
Let $\psi$ in $\Psi_d \subseteq \Psi_1$. With \citet[18.5.11]{dlmf} we obtain
\begin{align*}
\psi(\theta) &= \sum_{n=0}^\infty b_{n,1} \cos (n \theta) = \sum_{j=0}^\infty b_{j,d} \frac{C_j^{(d-1)/2}(\cos \theta)}{C_j^{(d-1)/2}(1)}\\
&= \sum_{j=0}^\infty b_{j,d} \sum_{l=0}^j \frac{\Gamma\big(\frac{d-1}{2}+l\big)\Gamma\big(\frac{d-1}{2} + j - l\big)}{l!(j-l)!\Gamma\big(\frac{d-1}{2}\big)^2} \frac{j!\Gamma(d-1)}{\Gamma(d-1+j)}\cos((j-2l)\theta) \\ 
&= \sum_{j=0}^\infty b_{j,d} \sum_{\substack{m=-j \\ m+j \ \text{even}}}^j \kappa_d(j,m)\cos(m\theta).
\end{align*}
Using that $\kappa_d(j,-m)\cos(-m\theta)=\kappa_d(j,m)\cos(m\theta)$, we obtain 
\begin{align*}
\psi(\theta)&=\sum_{j=0}^\infty b_{2j,d} \big( \kappa_d(2j,0) + 2 \sum_{m=1}^j \kappa_d(2j,2m) \cos(2m \theta)\big)
\\ & \ \ \ \ \ + \sum_{j=0}^\infty b_{2j-1,d} \, 2 \sum_{m=1}^{j} \kappa_d(2j-1,2m-1) \cos((2m-1) \theta)
\\ &= \sum_{j=0}^\infty b_{2j,d} \kappa_d(2j,0) + \sum_{m=1}^\infty \cos(2m \theta)\, 2\sum_{j=m}^\infty b_{2j,d} \kappa_d(2j,2m)
\\ & \ \ \ \ \ +\sum_{m=1}^\infty \cos((2m-1) \theta)\, 2 \sum_{j=m}^\infty b_{2j-1,d} \kappa_d(2j-1,2m-1).
\end{align*}
For $\psi \in \Psi_\infty \subseteq \Psi_1$, the claim follows analogously using \eqref{eq:psiunendlich}, 
\[
(\cos \theta)^j = \sum_{l=0}^j \frac{j!}{l!(j-l)!} \cos((j-2l)\theta),
\]
and $\tau(j,-m)\cos(-m\theta)=\tau(j,m)\cos(m\theta)$.
\end{proof}

\begin{lem}\label{lem:moments}
For $d \ge 2$ and $\ell \ge 0$, we have 
\begin{align*}
2\sum_{n=1}^j(2n)^\ell \kappa_d(2j,2n) &\sim c_d(\ell) j^\ell, \\
2\sum_{n=1}^j(2n-1)^\ell \kappa_d(2j-1,2n-1) &\sim c_d(\ell) j^\ell, 
\end{align*}
where $\sim$ means that the sequences are asymptotically equivalent as $j \to \infty$. Here, $c_d(\ell)$ is defined recursively by
\[
c_{d+2}(\ell) = \frac{d}{d-1}\Big(c_d(\ell) - \frac{1}{4}c_d(\ell+2)\Big),
\]
with $c_2(\ell) = 2^\ell \Gamma\big((\ell+1)/2\big)/\Gamma\big(\ell/2 + 1\big)$ and $c_3(\ell) = 2^{\ell}/(\ell+1)$.
\end{lem}

\begin{proof}
For $d \ge 2$ and $j,n \ge 1$, it is easy to check that
\begin{equation}
\kappa_{d+2}(2j,2n) = \frac{d}{d-1} \Big(\frac{d-1+2j}{d+2j} - \frac{(2n)^2}{(d-1 + 2j)(d+2j)}\Big)\kappa_d(2j,2n),\label{eq:drec}
\end{equation}
\begin{multline}
\kappa_{d+2}(2j-1,2n-1) \\= \frac{d}{d-1} \Big(\frac{d-2+2j}{d-1+2j} - \frac{(2n-1)^2}{(d-2 + 2j)(d-1+2j)}\Big)\kappa_d(2j-1,2n-1).\label{eq:drec2}
\end{multline}
For $d=2$, the claim is shown in Lemma \ref{lem:Faulhaberd2}.
For $d=3$, using Faulhaber's formula, we have as $j \to \infty$
\[
2\sum_{n=1}^j(2n)^\ell \kappa_{3}(2j,2n) = \frac{2}{2j+1}\sum_{n=1}^j (2n)^\ell \sim \frac{2^\ell}{\ell+1}j^\ell,
\]
and also
\[
2\sum_{n=1}^j(2n-1)^\ell \kappa_{3}(2j-1,2n-1) = \frac{2}{2j}\sum_{n=1}^j (2n-1)^\ell \sim \frac{2^\ell}{\ell+1}j^\ell.
\]
We will show the claims inductively over $d$. By \eqref{eq:drec} we have as $j \to \infty$
\begin{align*}
2\sum_{n=1}^j(2n)^\ell &\kappa_{d+2}(2j,2n)= \frac{d}{d-1}\Big(\frac{d-1+2j}{d+2j} 2\sum_{n=1}^j(2n)^\ell \kappa_{d}(2j,2n)\\ &\qquad\qquad\qquad\qquad- \frac{1}{(d-1+2j)(d+2j)}2\sum_{n=1}^j(2n)^{\ell+2} \kappa_{d}(2j,2n)\Big)\\
&\sim \frac{d}{d-1}\Big(\frac{d-1+2j}{d+2j}c_d(\ell)j^\ell - \frac{1}{(d-1+2j)(d+2j)}c_d(\ell+2)j^{\ell+2}\Big)\\
&\sim j^{\ell} \frac{d}{d-1}\Big(c_d(\ell) - \frac{1}{4}c_d(\ell+2)\Big) = c_{d+2}(\ell) j^\ell.
\end{align*}
The second claim follows with the same steps using \eqref{eq:drec2} instead of \eqref{eq:drec}.
\end{proof}

\begin{lem}\label{lem:Faulhaberd2}
For $\ell \ge 0$, we have as $j \to \infty$
\begin{align*}
2\sum_{n=1}^j(2n)^\ell \kappa_2(2j,2n) &\sim 2^\ell \frac{\Gamma\big(\frac{l}{2}+\frac{1}{2}\big)}{\sqrt{\pi}\Gamma\big(\frac{l}{2} + 1\big)} j^\ell, \\
2\sum_{n=1}^j(2n-1)^\ell \kappa_2(2j-1,2n-1) &\sim 2^\ell \frac{\Gamma\big(\frac{l}{2}+\frac{1}{2}\big)}{\sqrt{\pi}\Gamma\big(\frac{l}{2} + 1\big)} j^\ell.
\end{align*}
\end{lem}

\begin{proof}
We obtain with Gautschi's inequality \citep[5.6.4]{dlmf}
\begin{align*}
\frac{1}{j^{\ell}}2\sum_{n=1}^j(2n)^\ell \kappa_2(2j,2n) &= \frac{2^{\ell+1}}{\pi}\frac{1}{j^{\ell}}\sum_{n=1}^j n^\ell \frac{\Gamma(j-n+1/2)\Gamma(j+n+1/2)}{\Gamma(j-n+1)\Gamma(j+n+1)}\\
&\le \frac{2^{\ell+1}}{\pi}\frac{1}{j^{\ell}}\sum_{n=1}^{j-1}\frac{n^\ell}{\sqrt{j^2 - n^2}} + \frac{2^{\ell+1}}{\sqrt{\pi}}\frac{\Gamma(2j+1/2)}{\Gamma(2j+1)}\\
&= \frac{2^{\ell+1}}{\pi}\frac{1}{j}\sum_{n=1}^{j-1}\frac{\big(\frac{n}{j}\big)^\ell}{\sqrt{1 - \big(\frac{n}{j}\big)^2}} + \frac{2^{\ell+1}}{\sqrt{\pi}}\frac{\Gamma(2j+1/2)}{\Gamma(2j+1)}.
\end{align*}
Using Gautschi's inequality again, it follows that the last term in the above expression vanishes as $j \to 0$. Furthermore, Gautschi's inequality also yields
\begin{align*}
\frac{1}{j^{\ell}}2\sum_{n=1}^j(2n)^\ell \kappa_2(2j,2n) &\ge 
\frac{2^{\ell+1}}{\pi}\frac{(j+1)^{\ell-1}}{j^{\ell}}\sum_{n=1}^{j}\frac{\big(\frac{n}{j+1}\big)^\ell}{\sqrt{1 - \big(\frac{n}{j+1}\big)^2}}.
\end{align*}
Let $f(x):= x^\ell/\sqrt{1-x^2}$. By the Euler Maclaurin formula, we have
\begin{align*}
\frac{1}{j}\sum_{n=1}^{j-1}f\Big(\frac{n}{j}\Big) 
&= \int_0^{1-1/j}f(x)\d x + \frac{1}{2j}\Big(f\Big(\frac{j-1}{j}\Big)- f(0)\Big) + R_j,
\end{align*}
where, for some constant $C > 0$,
\[
|R_j| \le \frac{C}{j^3}\int_0^{j-1}\Big|f''\Big(\frac{x}{j}\Big)\Big|\d x = \frac{C}{j^2}\int_0^{1 - 1/j}|f''(x)|\d x.
\]
Some calculus using l'H\^opital's rule shows that the right-hand side of the above expression converges to zero as $j \to \infty$. Furthermore, $f((j-1)/j)/j \to 0$ as $j \to \infty$ and 
\[
\int_0^{1-1/j}f(x)\d x \to \frac{\sqrt{\pi}\Gamma\big(\frac{l}{2}+\frac{1}{2}\big)}{2\Gamma\big(\frac{l}{2} + 1\big)}, \quad \text{as $j \to \infty$.}
\]
The second claim follows from the inequalities
\[
\sum_{n=1}^{j-1}(2n)^\ell\kappa_2(2j,2n) \le \sum_{n=1}^j(2n-1)^{\ell}\kappa_2(2j-1,2n-1) \le \frac{j+1}{j+\frac{1}{2}}\sum_{n=1}^{j}(2n)^\ell\kappa_2(2j,2n),
\]
which are obtained using that $\kappa_2(2j-1,2n-1) = \kappa_2(2j,2n)(j+n)/(j+n-1/2) = \kappa_2(2j,2n-2)(j-n+1)/(j-n+1/2)$.
\end{proof}

\begin{lem}\label{lem:inftymoments}
We have
\begin{align*}
2\sum_{n=1}^j(2n)^2 \tau(2j,2n) &= 2\sum_{n=1}^j(2n-1)^2 \tau(2j-1,2n-1) = 2j,\\
2\sum_{n=1}^j(2n)^4 \tau(2j,2n) &= 2\sum_{n=1}^j(2n-1)^4 \tau(2j-1,2n-1) = 4j(3j-1).\\
\end{align*}
\end{lem}
\begin{proof}
For $\ell\ge 0$ and $j \ge 1$, we define
\[
p_{2\ell}(j) = 2\sum_{n=1}^j (2n)^{2\ell} \tau(2j,2n) = 2^{-2j+1}\sum_{n=1}^j \binom{2j}{j+n} (2n)^{2\ell}.
\]
Using that for $n \le j-1$, we have $\binom{2(j+1)}{j+1+n} = \binom{2j}{j+n-1}+2\binom{2j}{j+n}+\binom{2j}{j+n+1}$, we find
\begin{align}
p_{2\ell}(j+1) 
=2^{-2j-1}\Big[&\sum_{n=0}^{j}\binom{2j}{j+n}(2(n+1))^{2\ell} + 2\sum_{n=1}^{j}\binom{2j}{j+n}(2n)^{2\ell}\nonumber\\&+  \sum_{n=2}^{j}\binom{2j}{j+n}(2(n-1))^{2\ell}  \Big].\label{eq:pjp1}
\end{align}
For $\ell = 0$, \eqref{eq:pjp1} simplifies to
\begin{align*}
p_{0}(j+1)
&= p_0(j) - 2^{-2j-1}\Big[\binom{2j}{j+1} - \binom{2j}{j}\Big].
\end{align*}
By induction, it follows that $p_0(j) = 1 - 2^{-2j}\binom{2j}{j}$. For $\ell \ge 1$, \eqref{eq:pjp1} implies that
\begin{align*}
p_{2\ell}(j+1)&= \frac{1}{2}p_{2\ell}(j) + 2^{-2j-1}\sum_{n=1}^{j}\binom{2j}{j+n}\big[(2(n+1))^{2\ell}+(2(n-1))^{2\ell}\big]\\&\quad + 2^{-2j-1}\binom{2j}{j}2^{2\ell} \\
&= \frac{1}{2}p_{2\ell}(j) + 2^{-2j}\sum_{n=1}^{j}\binom{2j}{j+n}\sum_{\substack{t=0\\\text{$t$ even}}}^{2\ell}\binom{2\ell}{t}2^{2\ell-t}(2n)^t + 2^{-2j-1}\binom{2j}{j}2^{2\ell}\\
&= p_{2\ell}(j) + \frac{1}{2}\sum_{\substack{t=2\\\text{$t$ even}}}^{2\ell-2}\binom{2\ell}{t}2^{2\ell-t}p_t(j) + 2^{2\ell-1}
\end{align*}
using the formula for $p_0(j)$. For $\ell = 1$, the formula simplifies to $p_2(j+1) = p_2(j)+2$, which yields $p_2(j) = 2j$ using that $p_2(1) = 2$. 
For $\ell=2$, the claim follows by induction using the result for $\ell = 1$. The statements concerning sums over odd integers can be shown analogously. 
\end{proof}

\bibliographystyle{plainnat}
\bibliography{biblio}

\end{document}